\documentclass{amsart}
  \usepackage{amsthm,amsmath,amssymb,hyperref}
  \usepackage{graphicx,color,float}

  \newtheorem{theorem}{Theorem}[section]
  
  \newtheorem{lemma}[theorem]{Lemma}
  
  \theoremstyle{definition}

  \theoremstyle{remark}
  \newtheorem{remark}[theorem]{Remark}
  \numberwithin{equation}{section}

   \newcommand{\SL}{\operatorname{SL}}

   \newcommand{\bC}{{\mathbb C}}

   \newcommand{\bP}{{\mathbb P}}
   \newcommand{\bQ}{{\mathbb Q}}
   \newcommand{\bR}{{\mathbb R}}

   \newcommand{\bZ}{{\mathbb Z}}

\begin{document}

   \title{BCOV rings on elliptic curves and Eta function}
   \author{So Okada}

   \address{National Institute of Technology, Oyama College. Tochigi, Japan
   323-0806}

   \email{okada@oyama-ct.ac.jp}

   \maketitle

   \begin{abstract} 
     Associated Legendre functions of the first kind give a family of
     BCOV rings on elliptic curves.  We prove that the family is
     parametrized by $q$-exponents of the eta function $\eta(q^{24})$.
     Our method involves a classification of rational solutions of a
     Riccati equation under some constraints.
  \end{abstract}

  \section{Introduction}\label{sec:first-intro}
  In this paper, we parametrize a family of BCOV rings on elliptic
  curves by the eta function.  As such, this paper can be seen as a
  step forward on understanding meromorphic ambiguity on BCOV theory
  \cite{BCOV} by a modular form.

  BCOV rings \cite{Hos} have been introduced to study {\it BCOV
    holomorphic anomaly equations} of Bershadsky, Cecotti, Ooguri and
  Vafa \cite{BCOV}.  BCOV theory has gained much interest in
  mathematics and physics \cite{YamYau,Ali,AliLan,KanZho}.

  A major challenge of BCOV theory is meromorphic ambiguity to compute
  Gromov-Witten potentials.  For this, let us take $\Gamma=\langle
  \left(\begin{smallmatrix} 1 & 0 \\ -1 & 1
    \end{smallmatrix}\right),
  \left(\begin{smallmatrix}
      1 & 1 \\
      0 & 1
    \end{smallmatrix}\right) \rangle\subset \SL_{2}(\bZ)$
  and recall a finitely-generated $\Gamma$-invariant BCOV ring
  ${\mathcal R}_{BCOV}^{\Gamma}$ on elliptic curves.  This ring is
  fundamental in BCOV rings.  To define each ${\mathcal
    R}_{BCOV}^{\Gamma}$, we need to choose $r(x)\in \bC(x)$ that
  solves the following Riccati equation. This choice corresponds to
  the meromorphic ambiguity of the BCOV theory. For $x\in \bP^{1}$,
  $\lambda\in \bQ$, and the {\it Griffith-Yukawa coupling}
  $C_{x}=\frac{1}{(1-432 x)x}$, there is the Riccati equation:
   \begin{equation}\label{eq:Hosono-Riccati}
     r'(x)+C_{x}r^{2}(x)-60=\lambda C_{x}.
  \end{equation}

  For Legendre associated functions of the first and second kinds
  $P_{\alpha}^{\beta}(x)$ and $Q_{\alpha}^{\beta}(x)$ and $C\in
  \bP^{1}$, Equation \ref{eq:Hosono-Riccati} admits the general
  solution:
  \begin{align*}
   r(x,\lambda,C)&=\frac{1}{12}\left(
   5+4320x + (-5 + 12 \sqrt{\lambda})  
   \frac{C P_{\frac{5}{6}}^{2\sqrt{\lambda}}(-1+864x)+ 
   Q_{\frac{5}{6}}^{2\sqrt{\lambda}}(-1+864x)}{
   C P_{-\frac{1}{6}}^{2 \sqrt{\lambda}}( -1 + 864 x) + 
   Q_{-\frac{1}{6}}^{2 \sqrt{\lambda}}(-1 + 864 x)}\right).
  \end{align*}
  By taking $C\to \infty$, set
  \begin{align*}
   r(x,\lambda)&=\frac{1}{12}\left(
   -5 + 4320 x + (-5 + 12 \sqrt{\lambda}) 
   \frac{P_{\frac{5}{6}}^{2 \sqrt{\lambda}}(-1 + 864 x)}{ P_{-\frac{1}{6}}^{2 \sqrt{\lambda}}(-1 + 864 x)}\right).
  \end{align*}
  
  Let $R_{\infty}$ be the family of ${\mathcal
    R}_{BCOV}^{\Gamma}$ for all $r(x,\lambda)\in \bC(x)$.  Let
  $\chi(n)$ be the Dirichlet character of mod 12 such that $\chi(\pm
  1)=1$ and $\chi(\pm 5)=-1$.  We prove the following.
  \begin{theorem}\label{thm}
    The family $R_{\infty}$ of finitely-generated
    $\Gamma$-invariant BCOV rings on elliptic curves 
    is parametrized
    by the $q$-exponents of the eta function
    $\eta(q^{24})=\sum_{n=1}^{\infty} \chi(n) q^{n^2}= q - q^{25} -
    q^{49} + q^{121} + q^{169}\cdots$. Namely, $r(x,\lambda)\in
    \bC(x)$ if and only if $144\lambda=1, 25, 49, 121, 169, \cdots$,
    squares of numbers prime to $6$.
  \end{theorem}

  \section{Proofs}
  To study when $r(x,\lambda)\in \bC(x)$, let us first consider
  $r(x,\lambda)$ at a fixed singularity of Equation
  \ref{eq:Hosono-Riccati}.

  \begin{lemma}\label{lem:numbers}
    If $r(x,\lambda)\in \bC(x)$, then
    $144\lambda=1, 25, 49,
    121, 169, \cdots$.
  \end{lemma}
  \begin{proof}
    At $x=\infty$, unless $\frac{5}{6}-2\sqrt{\lambda}\in \bZ_{\leq
      0}$, the Lawrent expansion of $r(x,\lambda)$ is $$ 72 x-\frac{6
      \cdot 2^{\frac{2}{3}} \cdot \sqrt{\pi} \cdot
      \Gamma(\frac{5}{6}-2 \sqrt{\lambda})}{\Gamma(\frac{1}{6})\cdot
      \Gamma(\frac{1}{6}-2 \sqrt{\lambda})} x^{\frac{1}{3}}-
    \frac{1}{12}+O\left(\frac{1}{x}\right)^{\frac{1}{6}} $$ Since
    $r(x,\lambda)\in \bC(x)$,
    $\frac{1}{\Gamma(\frac{1}{6}-2\sqrt{\lambda})}$ has to
    vanish. Thus, $\frac{1}{6}-2\sqrt{\lambda}\in \bZ_{\leq 0}$.
  \end{proof}

  For $n,m\in \bR$, let us study 
  $f(n,m,x)=\frac{P_{n+1}^{m}(x)}{P_{n}^{m}(x)}$ when $m$ increases.

     \begin{lemma}\label{lem:induction}
       If $n\neq m$, we have
       \begin{align*}
         x-f(n,m+1,x)
         =       \frac{
           (n+m+1)(1-x^{2})}{ (n-m+1)f(n,m,x) -(n+m+1)x}. 
       \end{align*}
     \end{lemma}
     \begin{proof}

   Recall the three-term recurrences \cite[14.10.1,14.10.2]{DLMF}:
   \begin{align}
     (1-x^{2})^{\frac{1}{2}}P_{n}^{m+2}(x)+2(m+1)xP_{n}^{m+1}(x)
       =-(n-m)(n+m+1)(x^{2}-1)^{\frac{1}{2}}P_{n}^{m}(x), \label{eq6}\\
     (1-x^{2})^{\frac{1}{2}}P_{n}^{m+1}(x)-(n-m+1)P_{n+1}^{m}(x)
     =-(n+m+1)xP_{n}^{m}(x). \label{eq2}
   \end{align}
     
     By Equation \ref{eq2},  put
     \begin{align}
      (1-x^{2})^{\frac{1}{2}}P_{n}^{m+2}(x)-(n-m)P_{n+1}^{m+1}(x)
     +(n+m+2)xP_{n}^{m+1}(x)&=0. \label{eq3}
   \end{align}
   Let $F(n,m,x)=(n-m+1)f(x) -(n+m+1)x$. Then,
   by $P_{n+1}^{m}(x)=P_{n}^{m}(x)f(n,m,x)$ and Equation \ref{eq2},
   \begin{align}
     P_{n}^{m+1}(x)=
     (1-x^{2})^{-\frac{1}{2}} P_{n}^{m}(x)F(n,m,x). \label{eq4}
     \end{align}

     By Equations \ref{eq3} and \ref{eq4},
     \begin{equation}
       (1-x^{2})^{\frac{1}{2}}P_{n}^{m+2}(x)-(n-m)P_{n+1}^{m+1}(x) 
       =-(n+m+2)x     (1-x^{2})^{-\frac{1}{2}} P_{n}^{m}(x) F(n,m,x)
       \label{eq5}
     \end{equation}

     Thus, subtracting Equation \ref{eq6} from Equation \ref{eq5}
     gives
     \begin{align*}
       -2(m+1)xP_{n}^{m+1}(x)-(n-m)P_{n+1}^{m+1}(x)
       &= \\
       P_{n}^{m}(x)   (
       -(n+m+2)x     (1-x^{2})^{-\frac{1}{2}} F(n,m,x)
       +(n-m)(n+m+1)(1-x^{2})^{\frac{1}{2}})
       \end{align*}
       Thus,
       \begin{align*}
         -2(m+1)x-(n-m)\frac{P_{n+1}^{m+1}(x)}{P_{n}^{m+1}(x)}
       = \\
       -2(m+1)x-(n-m)f(n,m+1,x)
      =\\
       \frac{P_{n}^{m}(x)   }{P_{n}^{m+1}(x)}
       (
       -(n+m+2)x     (1-x^{2})^{-\frac{1}{2}} F(n,m,x)
       +(n-m)(n+m+1)(1-x^{2})^{\frac{1}{2}})
       \end{align*}
       Again, by Equation \ref{eq4},
       \begin{align*}
         -2(m+1)x-(n-m)f(n,m+1,x)
       = \\
       \frac{P_{n}^{m}(x)   
         (
       -(n+m+2)x     (1-x^{2})^{-\frac{1}{2}} F(n,m,x)
       +(n-m)(n+m+1)(1-x^{2})^{\frac{1}{2}})}{(x^{2}-1)^{-\frac{1}{2}} P_{n}^{m}(x)
       F(n,m,x)}
       = \\
       \frac{
            -(n+m+2)x   F(n,m,x)
         +(n-m)(n+m+1)(1-x^{2})}{ F(n,m,x)}     =\\
       -(n+m+2)x   
       + \frac{(n-m)(n+m+1)(1-x^{2})}{F(n,m,x)}.
     \end{align*}
     Thus, the lemma holds.
\end{proof}

Let us prove Theorem \ref{thm}.

\begin{proof}
  We confirm the converse of Lemma \ref{lem:numbers}.  Since
  $f(-\frac{1}{6},\frac{1}{6},x)=x$,
  $r(x,\frac{1}{144})=-\frac{1}{12}+72x\in \bC(x)$. Thus, by Lemma
  \ref{lem:induction}, $r(x,\frac{i^2}{144})\in \bC(x)$ for
  $i=1,7,13,\cdots$.  If $\lambda=\frac{5^2}{144}$, since $-5+12
  \sqrt{\lambda}=0$, $r(x,\lambda)=-\frac{5}{12} + 360 x\in \bC(x)$.
  For $\lambda=\frac{i^2}{144}$ of $i=11,17,23,\cdots$,
  $f(-\frac{1}{6},\frac{11}{6},x)=\frac{1}{x}$ implies
  $r(x,\lambda)\in \bC(x)$ by Lemma \ref{lem:induction}.  Thus, the
  assertion holds.
\end{proof}

\begin{remark}
  By Lemma \ref{lem:numbers}, we do not have to assume $\lambda\in
  \bQ$ to define ${\mathcal R}_{BCOV}^{\Gamma}$ for $r(x,\lambda)$.
  By the proof of the theorem and Lemma \ref{lem:induction}, we
  observe that $r(x,\lambda)\in \bC(x)$ implies $r(x,\lambda)\in
  \bQ(x)$.  Lemma \ref{lem:induction} holds for associated Legendre
  functions of the second kind. But, $r(x,\frac{1}{144},C)\not\in
  \bC(x)$ unless $C\to \infty$, since the Lawrent expansion of
  $r(x,\frac{1}{144},C)$ at $x=\infty$ is $72 x+\frac{4
    (-2)^{\frac{2}{3}}\Gamma(-\frac{4}{3})\Gamma(\frac{5}{6})}{\sqrt{3
      \pi } (2 i
    C+\pi)}x^{\frac{1}{3}}-\frac{1}{12}+O\left(\frac{1}{x}\right)^{\frac{1}{4}}$.
\end{remark}

\section*{Acknowledgments}
The author would like to thank Professors S. Hosono and Y. Ohyama for
their helpful communications.

\bibliographystyle{amsalpha}

\end{document}